\documentclass[10pt]{amsart}
\usepackage{amssymb,amsfonts}
\usepackage[all,arc]{xy}
\usepackage{enumitem}
\usepackage{mathrsfs}
\usepackage{ stmaryrd }
\usepackage{url}

\usepackage{fullpage}

\usepackage
[pdfauthor={Preston Wake and Carl Wang Erickson},
 pdftitle={Ordinary Pseudorepresentations and Modular Forms},
 bookmarks=false]
{hyperref}

\newcounter{qcounter}

\newcommand\define{\newcommand}

\define\isoto{\xrightarrow{\sim}}
\define\onto{\twoheadrightarrow}

\DeclareMathOperator{\Spec}{Spec}
\define\Ch{\mathrm{Char}_\Lambda}

\newcommand{\dia}[1]{{\langle #1 \rangle}}

\newcommand{\ttmat}[4]{\left( \begin{array}{cc}
#1 & #2 \\
#3 & #4
\end{array}
\right)}

\newcommand{\Z}{\mathbb{Z}}
\newcommand{\Q}{\mathbb{Q}}

\newcommand{\F}{\mathbb{F}}

\newcommand{\fH}{\mathfrak{H}}
\newcommand{\h}{\mathfrak{h}}

\newcommand{\sO}{\mathcal{O}}

\newcommand{\Lam}{\Lambda}
\newcommand{\I}{\mathcal{I}}
\newcommand{\p}{\mathfrak{p}}
\newcommand{\m}{\mathfrak{m}}

\newcommand{\cJ}{\mathcal{J}}

\newcommand{\chii}{{\chi^{-1}}}

\newcommand{\Hom}{\mathrm{Hom}}
\newcommand{\Gal}{\mathrm{Gal}}

\newcommand{\Aut}{\mathrm{Aut}}

\newcommand{\Ext}{\mathrm{Ext}}

\newcommand{\End}{\mathrm{End}}

\newcommand{\X}{\mathfrak{X}}

\newcommand{\lb}{{[\![}}
\newcommand{\rb}{{]\!]}}

\newcommand{\red}{\mathrm{red}}

\define\ord{{\mathrm{ord}}}
\define\Cl{{\mathrm{Cl}}}
\define\GL{{\mathrm{GL}}}
\define\kcyc{\kappa_{\mathrm{cyc}}}

\define{\Fitt}{\mathrm{Fitt}}
\define{\Ann}{\mathrm{Ann}}

\define\RGammac{{\mathrm{R}\Gamma_{(c)}}}
\define\RGamma{{\mathrm{R}\Gamma}}
\define\RHom{{\mathrm{RHom}}}

\define\Hc{H_{(c)}}

\newtheorem{thm}{Theorem}[subsection] 
\newtheorem*{thm*}{Theorem}

\newtheorem{cor}[thm]{Corollary}
\newtheorem{prop}[thm]{Proposition}
\newtheorem{lem}[thm]{Lemma}

\theoremstyle{definition}
\newtheorem{defn}[thm]{Definition}

\newtheorem{eg}[thm]{Example}

\theoremstyle{remark}
\newtheorem{rem}[thm]{Remark}

\newcommand{\ra}{\rightarrow}

\newcommand{\lra}{\longrightarrow}
\newcommand{\lrisom}{\buildrel\sim\over\lra}
\newcommand{\risom}{\buildrel\sim\over\ra}
\newcommand{\rinj}{\hookrightarrow}

\newcommand{\rsurj}{\twoheadrightarrow}

\newcommand{\bF}{\mathbb{F}}
\newcommand{\bQ}{\mathbb{Q}}
\newcommand{\bZ}{\mathbb{Z}}

\newcommand{\cO}{\mathcal{O}}
\newcommand{\Db}{{\bar D}}

\usepackage{color}

\newcommand{\lr}{{\langle-\rangle}}
\newcommand{\tr}{{\mathrm{tr}}}

\newcommand{\sm}[4]{\ensuremath{\big(\begin{smallmatrix}#1 & #2 \\ #3 & #4\end{smallmatrix}\big)}}

\newcommand{\GMA}{{\mathrm{GMA}}}

\makeatletter
\let\c@equation\c@thm
\makeatother
\numberwithin{equation}{subsection}

\title{Ordinary Pseudorepresentations and Modular Forms}

\author{Preston Wake}
\address{UCLA Mathematics Department \\
Box 951555 \\
Los Angeles, CA 90095-1555 }
\email{wake@math.ucla.edu}

\author{Carl Wang Erickson}
\address{Department of Mathematics, Imperial College London \\
	London SW7 2AZ, UK}
\email{c.wang-erickson@imperial.ac.uk}

\begin{document}

\maketitle

\begin{abstract}
In this short note, we observe that the techniques of \cite{WWE2015} can be used to provide a new proof of some of the residually reducible modularity lifting results of Skinner and Wiles \cite{SW1999}. In these cases, we have found that a deformation ring of ordinary pseudorepresentations is equal to the Eisenstein local component of a Hida Hecke algebra. We also show that Vandiver's conjecture implies Sharifi's conjecture. 
\end{abstract}

\tableofcontents

\section{Introduction}
The key technical innovation behind our previous work \cite{WWE2015} was our definition of an \emph{ordinary 2-dimensional pseudorepresentation} of $G_\bQ = \Gal(\bar \bQ/\bQ)$. Using this notion, we were able to study ordinary Galois deformations in the case where the residual representation is reducible. In particular, we constructed a universal ordinary pseudodeformation ring $R^\ord_\Db$ with residual pseudorepresentation $\Db$. We also showed that the Galois action on the Eisenstein part of the cohomology of modular curves gives rise to an ordinary pseudorepresentation valued in the Eisenstein component $\fH$ of the ordinary Hecke algebra. Studying deformations of ordinary pseudorepresentations, we showed that, if Greenberg's conjecture holds, then certain characteristic $0$ localizations $\fH_\p$ of $\fH$ are Gorenstein. Under the same assumption, we also proved an isomorphism $(R^\ord_\Db)_\p \isoto \fH_\p$. 

In this note, we show that the methods of \cite{WWE2015} can be extended to the whole Eisenstein component $\fH$, provided that we make stronger assumptions on class groups. Namely, we have to assume that the plus-part $X^+$ of the Iwasawa class group of the relevant cyclotomic field vanishes. When the tame level $N$ is $1$, then this is known as Vandiver's conjecture. When $N>1$, there are examples where $X^+ \ne 0$, but it is still often the case that $X^+=0$.
Assuming $X^+=0$, we get an isomorphism $R^\ord_\Db \risom \fH$ (Theorem \ref{thm:R=T}). As a consequence, we have a new technique to establish the residually reducible ordinary modularity theorem of Skinner and Wiles \cite{SW1999} over $\bQ$, in some cases (Theorem \ref{thm: our sw}). We also derive new results on Gorensteinness of Hecke algebras (Corollary \ref{cor: principal}) and prove new results toward Sharifi's conjecture (Corollary \ref{cor: sharifi}). In particular, we prove that $\fH$ is Gorenstein when $X^+ = 0$, an implication that was known previously only after assuming Sharifi's conjecture \cite{wake1}. Previous partial results in this direction by Skinner-Wiles \cite{SW1997} and Ohta \cite{ohta2005} require much stronger conditions on class groups.

As well as proving these new results, we review the most novel parts of \cite{WWE2015}. In this way, this note may be serve as an introduction to \cite{WWE2015}. 

\subsection{Ordinary pseudorepresentations} 
A $2$-dimensional pseudorepresentation of $G_\Q$ with values in a ring $A$ is the data of two functions $\{\tr, \det\}$ that satisfy conditions as if they were the trace and determinant of a representation $G_\Q \to \GL_2(A)$. The (fine) moduli of pseudorepresentations may be thought of as the coarse moduli of Galois representations produced by geometric invariant theory \cite[Thm.\ A]{WE2015}. In this respect, our results suggest that coarse moduli rings of Galois representations are most naturally comparable with Hecke algebras. Indeed, most previous $R=T$ theorems have been established where $R$ is a deformation ring for a residually irreducible Galois representation, in which case the fine and coarse moduli of Galois representations are identical.

The ordinary condition is somewhat subtle when applied to pseudorepresentations. For example, if one thinks about the case when $A$ is a field, a representation $\rho: G_\bQ \ra \GL_2(A)$ is defined to be ordinary when $\rho\vert_{G_{\Q_p}}$ is reducible with a twist-unramified quotient. While $\{\tr \rho \vert_{G_{\Q_p}}, \det \rho\vert_{G_{\Q_p}}\}$  knows nothing about which of the two Jordan-H\"older factors is the quotient, $D_\rho = \{\tr\rho, \det\rho\}$ can often distinguish them. This allows for the definition of an ordinary pseudorepresentation of $G_\bQ$, which we extend to non-field coefficients. We overview this and other background from \cite{WWE2015} in \S\S2-3. 

\subsection{Outline of the proof} The \'etale cohomology of compactified modular curves defines a $G_\bQ$-module $H$ over the cuspidal quotient $\h$ of $\fH$. However, $H$ is a representation (i.e.~locally free $\h$-module) if and only if $\h$ is Gorenstein, which is not always true. Nonetheless, $H$ always induces an ordinary $\h$-valued pseudorepresentation deforming the residual pseudorepresentation $\Db$. This pseudorepresentation extends to $\fH$, resulting in a surjection $R^\ord_\Db \rsurj \fH$. 

This map is naturally a morphism of augmented $\Lambda$-algebras, where $\Lambda$ is an Iwasawa algebra. The augmentation ideals
\[
\I := \ker(\fH \rsurj \Lambda), \qquad  \cJ := \ker(R^\ord_\Db \rsurj \Lambda)
\]
correspond to the Eisenstein family of $\Lambda$-adic modular forms and the reducible locus of Galois representations, respectively. We can show that certain Iwasawa class groups surject onto $\cJ/\cJ^2$, which is the cotangent module relative to the reducible family. The Vandiver-type assumption $X^+=0$ is used to show that one of the relevant Iwasawa class groups is cyclic. Using a version of Wiles's numerical criterion \cite[Appendix]{wiles1995}, with the class groups playing the role of Wiles's $\eta$, we can show that this forces $R^\ord_\Db \onto \fH$ to be an isomorphism.

One novel aspect of this proof is that we are able to control a ``cotangent space'' $\cJ/\cJ^2$ of a pseudodeformation ring in terms of Galois cohomology. Moreover, we use Galois cohomology groups with coefficients in $\Lambda$, while the usual approach works over a field. Such control is critical to proving $R=T$ theorems in the residually irreducible case, and $R=T$ theorems for pseudorepresentations were lacking because this control was not as available. In our situation, the relevant Galois cohomology is determined by class groups. 

\subsection{Acknowledgements} The authors would like to thank Romyar Sharifi and Eric Urban for helpful conversations. We also thank the referee for many helpful comments and corrections. Both authors would like to recognize the Simons Foundation for support in the form of AMS-Simons travel grants. Preston Wake was supported by the National Science Foundation under the Mathematical Sciences Postdoctoral Research Fellowship No.~1606255.

\section{Background: Iwasawa theory and Hecke algebras}

This section is a brief synopsis of Sections 2, 3 and 6 of \cite{WWE2015}. We overview background information from Iwasawa theory and ordinary $\Lambda$-adic Hecke algebras and modular forms.

\subsection{Iwasawa algebra and Iwasawa modules} 
\label{subsec:iwasawa}
We review Section 2 of \cite{WWE2015}. 

Let $p \ge 5$ be a prime number, and let $N$ be an integer such that $p \nmid N\phi(N)$. Let
$$
\theta: (\Z/Np\Z)^\times \to \overline{\Q}_p^\times
$$
be an even character. Let $\chi=\omega^{-1}\theta$, where 
$$
\omega: (\Z/Np\Z)^\times \to (\Z/p\Z)^\times \to \Z_p^\times
$$
is the Teichm\"uller character. Our assumption on $N$ implies that each of these characters is a Teichm\"uller lift of a character valued in a field extension $\bF$ of $\bF_p$. By abuse of notation, we also use $\theta,\chi,\omega$ to refer to these characters. 

We assume that $\theta$ satisfies the following conditions:
\begin{enumerate}[label=(\alph*)]
\item $\theta$ is primitive,
\item if $\chi|_{(\Z/p\Z)^\times}=1$, then $\chi|_{(\Z/N\Z)^\times}(p) \ne 1$, and 
\item if $N=1$, then $\theta \ne \omega^2$.
\end{enumerate}
A subscript $\theta$ or $\chi$ on a module refers to the eigenspace for an action of $(\Z/Np\Z)^\times$. A superscript $\pm$ will denote the $\pm 1$-eigenspace for complex conjugation. Let $S$ denote the set of primes dividing $Np$ along with the infinite place, and let $G_{\Q,S}$ be the unramified outside $S$ Galois group. We fix a decomposition group $G_p \subset G_{\Q,S}$ and let $I_p\subset G_p$ denote the inertia subgroup. Let $\kcyc$ denote the $p$-adic cyclotomic character. 

Fix a system $(\zeta_{Np^r})$ of primitive $Np^r$-th roots of unity such that $\zeta_{Np^{r+1}}^p=\zeta_{Np^r}$ for all $r$. Let $\Q_\infty= \Q(\zeta_{Np^\infty})$ and let $\Gamma = \Gal(\Q_\infty/\Q(\zeta_{Np}))$.

Let $\Cl(\Q(\zeta_{Np^r}))$ be the class group, and let 
$$
X=\varprojlim \Cl(\Q(\zeta_{Np^r}))\{p\}.
$$
There is action of $\Gamma$ on $X$. By class field theory, $X=\Gal(L/\Q_\infty)$ where $L$ is the maximal pro-$p$, abelian, unramified extension of $\bQ_\infty$. A closely related object is $\X=\Gal(M/\Q_\infty)$ where $M$ is the maximal pro-$p$ abelian extension unramified outside $Np$.

Let $\Z_{p,N} = \varprojlim \Z/Np^r\Z$.  Let $\Lambda_\theta = \Z_p\lb\Z_{p,N}^\times\rb_\theta$. We write $\Lam$ for $\Lam_\theta$ when $\theta$ is implicit. This $\Lam$ is a local component of the semilocal ring $\Z_p\lb\Z_{p,N}^\times\rb$ and is abstractly isomorphic to $\sO\lb\Gamma\rb \simeq \sO\lb T\rb$, where $\cO$ is the extension of $\bZ_p$ generated by the values of $\theta$. Notice that the action on $(\zeta_{Np^r})$ gives an isomorphism $\Gamma \simeq \ker(\Z_p^\times \to (\Z/p\Z)^\times)$. 

Let $M \mapsto M^\#$ and $M \mapsto M(r)$ be the functors on $\Z_p\lb\Z_{p,N}^\times\rb$-modules as defined in \cite[\S2.1.3]{wake2}. Namely, $M^\#=M(r)=M$ as $\Z_p$-modules, but $\gamma \in \Z_{p,N}^\times$ acts on $M^\#$ (resp. $M(r)$) as $\gamma^{-1}$ (resp. $\kcyc^r(\gamma)\gamma$) acts on $M$. We sometimes, especially when using duality, are forced to consider $\Z_p\lb\Z_{p,N}^\times\rb$-modules with characters other than $\theta$, but we use these functors to make the actions factor through $\Lam$ so we can treat all modules uniformly.

We define $\xi_\chi \in \Lam$ to be a generator of the principal ideal $\Ch(X_\chi(1))$. By the Iwasawa Main Conjecture, it may be chosen to be a power series associated to a Kubota-Leopoldt $p$-adic $L$-function.

Consider the $\Lambda$-valued character $\dia{-}:G_{\Q,S} \onto \Gamma \subset \Lambda^\times$, where $G_{\Q,S} \onto \Gamma$ is the quotient map. We define $\Lambda[G_{\Q,S}]$-modules $\Lambda^\dia{-}$ and $\Lambda^\#$ to be $\Lambda$ with $G_{\Q,S}$ acting by $\dia{-}$ and $\dia{-}^{-1}$, respectively. 

\subsection{Duality and consequences}  
\label{subsec: duality}
We review some relevant parts of Section 6 of \cite{WWE2015}. To compare conditions on various class groups, we use the following $\Lambda$-adic version of Poitou-Tate duality. It is a generalization of \cite[Prop.\ 6.2.1]{WWE2015}.

Here, $K$ is a number field, and $U$ is an open dense subset of $\Spec(O_K[1/p])$. The compactly supported cohomology $H_{(c)}^\bullet(U,-)$ is defined to be the cohomology of
\[
\mathrm{Cone}\left(C^\bullet(U,-) \to \bigoplus_{v \not \in U} C^\bullet(K_v,-)\right)
\]
where $C^\bullet(U,-)$ and $C^\bullet(K_v,-)$ are the standard complexes that compute Galois cohomology.

\begin{prop}\label{prop:Lambda Poitou-Tate duality}
Let $T$ be a finitely generated projective $\Lambda$-module equipped with a continuous action of $G_K$, unramified at places outside $U$, and let $M$ be a finitely generated $\Lam$-module. Then there is a quasi-isomorphism
\begin{equation*}
\RGammac(U,T \otimes_\Lam M) \lrisom \RHom_\Lambda(\RGamma(U,T^*(1)),M)[-3]
\end{equation*}
that is functorial in $M$. There is a similar quasi-isomorphism when $\RGammac$ and $\RGamma$ are swapped, i.e.
\begin{equation*}
\RGamma(U,T \otimes_\Lam M) \lrisom \RHom_\Lambda(\RGammac(U,T^*(1)),M)[-3].
\end{equation*}
Here $T^*$ is the dual representation $\Hom(T,\Lambda)$.
\end{prop}

\begin{proof}
We prove the first quasi-isomorphism. The proof of the second is similar.

For the case where $M = \Lam$, see \cite[Prop.\ 5.4.3, pg.\ 99]{nekovar2006} or \cite[\S1.6.12]{FK2012}. Then we have quasi-isomorphisms
\begin{align*}
\RGammac(U,T)\otimes^\mathbb{L}_\Lam M & \isoto \RHom_\Lambda(\RGamma(U,T^*(1)),\Lambda)[-3] \otimes^\mathbb{L}_\Lam M \\
& \isoto \RHom_\Lambda(\RGamma(U,T^*(1)),M)[-3]
\end{align*}
where the first comes from the $M = \Lam$ case, and the second is standard (for example \cite[Exer.\ 10.8.3]{weibel1994}). To prove the proposition, we are reduced to producing a quasi-isomorphism
\begin{equation}
\label{eq: tensor in and out}
\RGammac(U,T)\otimes^\mathbb{L}_\Lam M \isoto \RGammac(U,T \otimes_\Lam M).
\end{equation}
This follows from \cite[Proposition 3.1.3]{LS2013} (and its compactly supported analog, which, as remarked in the proof of Proposition 4.1.1 of \emph{loc.~ cit.}, can be established similarly).
\end{proof}

The proposition yields spectral sequences with second page 
\begin{equation}
\label{eq: spectral Rc=RHom(R)}
E_2^{i,j}=\Ext_\Lam^i(H^{3-j}(U,T^*(1)), M) \implies \Hc^{i+j}(U,T \otimes_\Lam M),
\end{equation}
\begin{equation}
\label{eq: R = RHom(Rc)}
E_2^{i,j}=\Ext_\Lam^i(\Hc^{3-j}(U,T^*(1)), M) \implies H^{i+j}(U,T \otimes_\Lam M).
\end{equation}
These spectral sequences are functorial in $M$.

We record the influence of the assumption that $X_\theta = 0$. In the proof, we make use of the following lemma on the structure of $\Lambda$-modules.

\begin{lem}
\label{lem: Lambda-modules}
Let $M$ be a finitely generated $\Lambda$-module. Say that $M$ is type 0 if $M$ is free, type 1 if $M$ is torsion and has projective dimension $1$, and type 2 if $M$ is finite. Then $M$ is type $i$ if and only if $\Ext^j_\Lambda(M,\Lambda)=0$ for all $j \ne i$. Moreover, $M$ is type 1 if and only if $M$ is torsion and has no non-zero finite submodule.
\end{lem}
\begin{proof}
See \cite[\S3]{jannsen1989}.
\end{proof}

The following is well-known to experts.

\begin{prop}
\label{prop:X_control}
$X_\theta = 0$ if and only if $\X_\chii^\#(1)$ is a free $\Lambda$-module of rank 1.
\end{prop}
\begin{proof}
As in \cite[Cor.\ 6.3.1]{WWE2015}, we have $X_\theta=H^2(\Z[1/Np],\Lambda^\#(1))$. Since $\X$ is the Pontryagin dual of $H^1(\Z[1/Np],\Q_p/\Z_p)$, (classical) Poitou-Tate duality implies that
 \[ 
 \X = \varprojlim \Hc^2(\Z[1/Np, \zeta_{Np^r}],\Z_p(1)).
\]
We can then deduce $\X_\chii^\#(1)=H^2_{(c)}(\Z[1/Np],\Lambda^\dia{-})$, as in \cite[Cor.\ 6.1.3]{WWE2015}. 

Analyzing spectral sequence \eqref{eq: R = RHom(Rc)} above with $T=\Lambda^\dia{-}$ and $M=\Lambda$, we see that $E_2^{i,j}=0$ for cohomological dimension reasons unless $i,j \in \{0,1,2\}$. We find
\[
\Ext^1_\Lam(\X_\chii^\#(1), \Lam) = X_\theta,\qquad \Ext^2_\Lam(\X_\chii^\#(1), \Lam) = 0.
\] 
Then $\X_\chii^\#(1)$ is a free $\Lambda$-module if and only if $X_\theta = 0$ by Lemma \ref{lem: Lambda-modules}.

The fact that the rank is then 1 follows from class field theory and Iwasawa's theorem. Indeed, class field theory implies that there is an exact sequence
\[
0 \to U_\chii^\#(1) \to \X_\chii^\#(1) \to X_\chii^\#(1) \to 0,
\]
where $U$ is an Iwasawa local unit group at $p$, and Iwasawa's theorem implies that $U_\chii^\#(1)$ is free of rank 1 over $\Lambda$ (see \cite[\S2.1]{WWE2015} and the references given there -- note that there is no contribution from the local units at primes dividing $N$ because $\theta$ is primitive). Since $X_\chii^\#(1)$ is $\Lambda$-torsion, this implies that $\X_\chii^\#(1)$ has rank $1$.
\end{proof}

\subsection{Hecke algebras} We review Section 3 of \cite{WWE2015}. Let 
\[
\tilde{H}'=\varprojlim H^1(Y_1(Np^r),\Z_p)_\theta^\ord, \ H'=\varprojlim H^1(X_1(Np^r),\Z_p)_\theta^\ord
\]
where the subscript $\theta$ denotes the eigenspace for the diamond operators. Let $\fH'$ and $\h'$ denote the Hida Hecke algebras acting on $\tilde{H}'$ and $H'$, respectively. There is a unique maximal ideal of $\fH'$ containing the Eisenstein ideal for $\theta$; let $\fH$ and $\h$ be the localizations of $\fH'$ and $\h'$ at the Eisenstein maximal ideal, and let $\tilde{H}=\tilde{H}' \otimes_{\fH'} \fH$ and $H = H' \otimes_{\h'} \h$. Let $\I \subset \fH$ be the Eisenstein ideal, and let $I \subset \h$ be the image of $\I$.

By Hida theory, each of $\tilde{H}$, $H$, $\fH$ and $\h$ is finite and flat over $\Lambda$. There are also canonical isomorphisms of $\fH$-modules $\fH/\I \cong \Lambda$, $\h/I \cong \Lambda/\xi_\chi$ and $\I \cong I$ (see \cite[Prop.\ 3.2.5]{WWE2015}), making $\fH$ an augmented $\Lambda$-algebra.

\section{Ordinary Pseudorepresentations}
\label{sec:OP}

We define ordinary pseudorepresentations and show that they are representable by an ordinary pseudodeformation ring $R^\ord_\Db$, recapping results of \cite{WWE2015}. In particular, we will review background on pseudorepresentations, Cayley-Hamilton algebras, and generalized matrix algebras from \S5 of \emph{loc.~cit.}  

We highlight the following important points:
\begin{itemize}[leftmargin=2em]
\item The definition is ``not local,'' in the sense that it does not have the form  ``$D: G_{\Q,S} \to A$ is ordinary if $D|_{G_{p}}$ is ordinary.'' 
\item When $A$ is a field, we can say that $D$ is ordinary if there exists an ordinary $G_{\bQ,S}$-representation $\rho$ such that $D$ is induced by $\rho$. 
\item While not every pseudorepresentation comes from a representation, we fix this problem by broadening the category of representations to include generalized matrix algebra-valued representations (\emph{GMA representations}). We first define ordinary GMA representations, and then say a pseudorepresentation is ordinary when there exists an ordinary GMA representation inducing it.
\end{itemize}

We fix some notation.  We use the letter $\psi$ to denote the functor that associates to a representation its induced pseudorepresentation. Let $\Db = \psi(\omega^{-1} \oplus \theta^{-1})$, which is the $\bF$-valued residual pseudorepresentation induced by the Galois action on $H$. Write $R_\Db$ for the pseudodeformation ring for $\Db$ \cite[\S5.4]{WWE2015} with universal object $D^u_\Db: G_{\bQ,S} \ra R_\Db$. In this section, $A$ will denote a Noetherian local $W(\bF)$-algebra with residue field $\F$. If $a \in A$, then $\bar{a} \in \F$ denotes the image of $a$.

\subsection{Representations valued in generalized matrix algebras} 

As in \cite[\S1.3]{BC2009}, we say that a \emph{generalized matrix algebra over $A$} is an associative $A$-algebra $E$ such that 
\[
E \lrisom \begin{pmatrix}
A & B \\ C & A\end{pmatrix}.
\]
This means that we fix an isomorphism $E \risom A \oplus B \oplus C \oplus A$ as $A$-modules for some $A$-modules $B$ and $C$, and there is an $A$-linear map $B \otimes_A C \ra A$ such that the multiplication in $E$ is given by 2-by-2 matrix multiplication. In this case, $A$ is called the \emph{scalar subring} of $E$ and $E$ is called an $A$-GMA. 

A \emph{GMA representation} with coefficients in $A$ and residual pseudorepresentation $\Db$ is a homomorphism $\rho: G_{\Q,S} \to E^\times$, such that $E$ is an $A$-GMA, and such that in matrix coordinates, $\rho$ is given by
\begin{equation}
\label{eq: GMA defm}
\sigma \mapsto \begin{pmatrix}
\rho_{1,1}(\sigma) & \rho_{1,2}(\sigma) \\ \rho_{2,1}(\sigma) & \rho_{2,2}(\sigma)\end{pmatrix}
\end{equation}
with $\overline{\rho_{1,1}(\sigma)}=\omega^{-1}(\sigma)$, $\overline{\rho_{2,2}(\sigma)}=\theta^{-1}(\sigma)$, and $\overline{\rho_{1,2}(\sigma)\rho_{2,1}(\sigma)}=0$. We emphasize the fact that we fix the order of the diagonal characters (this is slightly non-standard but will simplify our notation later).  

Given such a $\rho$, there is an induced $A$-valued pseudorepresentation, denoted $\psi_{\GMA}(\rho): G_{\bQ,S} \ra A$, given by $\tr(\rho)=\rho_{1,1}+\rho_{2,2}$ and $\det(\rho)=\rho_{1,1}\rho_{2,2}-\rho_{1,2}\rho_{2,1}$. 

\subsection{Universality} 
\label{subsec:universality}
A \emph{Cayley-Hamilton representation} over $A$ with residual pseudorepresentation $\Db$ is the data of a pair $(\rho: G_{\Q,S} \to E^\times, D:E \to A)$, where $E$ is an associative $A$-algebra such that $D \circ \rho$ is a pseudorepresentation deforming $\Db$. These data must satisfy an additional Cayley-Hamilton condition that, for all $x \in E$, $x$ must satisfy the characteristic polynomial associated to $x$ by $D$. If $\rho: G_{\Q,S} \to E^\times$ is a GMA representation, then $(\rho, \psi_{\GMA}(\rho))$ is a Cayley-Hamilton representation.

For our purposes, the important properties of Cayley-Hamilton representations are the following (see \cite[Prop.\ 3.2.2]{WE2015}). 
\begin{itemize}[leftmargin=2em]
\item There is a universal Cayley-Hamilton representation $(\rho^u: G_{\bQ,S} \ra E_\Db^\times, D^u: E_\Db \to R_\Db)$ with residual pseudorepresentation $\Db$, and the induced pseudorepresentation $D^u \circ \rho^u$ of $G$ is equal to the universal deformation of $\Db$. 
\item $E_\Db$ is finite as an $R_\Db$-module, and $\rho^u$ is continuous for the natural adic topology from $R_\Db$ on $E_\Db$. 
\item $E_\Db$ admits a unique $R_\Db$-GMA structure such that $\rho^u$ is a continuous GMA representation over $R_\Db$. 
\end{itemize}
In particular, any Cayley-Hamilton representation with residual pseudorepresentation $\Db$ is a GMA representation with residual pseudorepresentation $\Db$.

We will write 
\begin{equation}
\label{eq:CH-GMA}
E_\Db \cong \begin{pmatrix} R_\Db & B^u \\ C^u & R_\Db\end{pmatrix}, 
\end{equation}
for the decomposition of $E_\Db$ as in \eqref{eq: GMA defm}, and write $\rho_{i,j}^u$ for the corresponding coordinates of $\rho^u$. Similarly, for any GMA representation $\rho$ deforming $\Db$, we will write $\rho_{i,j}$ for the induced coordinate decomposition. 

\subsection{Reducibility}
\label{subsec:reducibility}
It will be important to understand the notion of a \emph{reducible pseudorepresentation} and the \emph{reducibility ideal} in $A$ for an $A$-valued pseudodeformation $D: G_{\bQ,S} \ra A$ of $\Db$. We call $D$ \emph{reducible} if $D=\psi(\chi_1 \oplus \chi_2)$ for characters $\chi_i: G_{\bQ,S} \ra A^\times$ such that $\bar{\chi}_1=\omega^{-1}$ and $\bar{\chi}_2=\theta^{-1}$. Otherwise, $D$ is called \emph{irreducible}. Equivalently, $D$ is reducible if $D=\psi_{\GMA}(\rho)$ for some GMA representation $\rho$ with scalar ring $A$ such that $\rho_{1,2}(G_{\bQ,S}) \cdot \rho_{2,1}(G_{\bQ,S})$ is zero.

Since  $\rho^u_{1,2}(G_{\bQ,S})$ and $\rho^u_{2,1}(G_{\bQ,S})$ generate $B^u$ and $C^u$, respectively, as $R_\Db$-modules, $D$ is reducible exactly when the image of $B^u \otimes_{R_\Db} C^u$ in $R_\Db$ under multiplication vanishes under $R_\Db \ra A$. Consequently, we call the image of $B^u \otimes_{R_\Db} C^u$ in $R_\Db$ the \emph{reducibility ideal} of $R_\Db$, and its image in $A$ is the reducibility ideal for $D$. 

\subsection{Ordinary GMA representations}
\label{subsec:ord_GMA_reps}

We will say that a representation of $\Gal(\overline{\bQ}_p/\bQ_p)$ on a $2$-dimensional $p$-adic vector space $V$ is {\em ordinary} if there exists a $1$-dimensional quotient representation $V \onto W$ such that $W(1)$ is unramified. A representation $\rho$ of $G_{\Q,S}$ is {\em ordinary} if $\rho|_{G_p}$ is ordinary. This notion of ordinariness is relatively restrictive compared to other uses of the term, but it will suit our purpose of studying Galois representations associated to ordinary modular forms.

Relative to the ordering of factors in \eqref{eq: GMA defm}, we have the following definition of an ordinary GMA representation. 
\begin{defn}
\label{defn:ord_C-H}
Let $\rho: G_{\bQ,S} \ra E^\times$ be a GMA representation with scalar ring $A$ and induced pseudorepresentation $\Db$. We call $\rho$ \emph{ordinary} provided that  
\begin{enumerate}
\item $\rho_{1,2}(G_p) = 0$, and
\item $ \rho_{1,1}\vert_{I_p} \simeq \kcyc^{-1} \otimes_{\bZ_p} A$
\end{enumerate}
where $\kcyc$ is the $p$-adic cyclotomic character.
\end{defn}

\begin{rem}
\label{rem: condition (b) is needed}
The condition that $\omega$ and $\theta$ are locally $p$-distinguished, i.e.~$\omega\vert_{G_p} \neq \theta\vert_{G_p}$, is critically necessary to making this definition sensible. This is equivalent to the assumption (b) of \S\ref{subsec:iwasawa}. 
\end{rem}

\begin{eg}
\label{eg:ordinary representation}
Let $\rho: G_{\Q,S} \to GL_2(A)$ be a representation with induced pseudorepresentation $\Db$. Then $\rho$ is ordinary if and only if there is a quotient character $\alpha$ of $\rho|_{G_p}$ such that $\bar{\alpha}=\omega^{-1}$ and $\alpha|_{I_p}=(\kcyc^{-1} \otimes_{\Z_p} A)|_{I_p}$. Indeed, writing $V_\rho=A^2$ with $G_{\Q,S}$ acting via $\rho$, we see that $\rho$ induces an $A$-algebra homomorphism $E_\Db \to \End_A(V_\rho)$. Writing $\rho$ in terms of the resulting coordinate decomposition, we have
\[
\rho = \ttmat{\rho_{1,1}}{\rho_{1,2}}{\rho_{2,1}}{\rho_{2,2}}.
\]
We see that $\bar{\rho}_{1,1}=\omega^{-1}$, and that $\rho$ is ordinary if and only if $\alpha:=\rho_{1,1}|_{G_p}$ is a quotient of $\rho|_{G_p}$ satisfying $\alpha|_{I_p}=(\kcyc^{-1} \otimes_{\Z_p} A)|_{I_p}$.

This definition is slightly more restrictive than the definition of ``ordinary representation" given by some authors. Nonetheless, our definition can be useful for studying those more general representations (see \S \ref{subsec:Skinner-Wiles}).
\end{eg}

\begin{eg}
\label{eg:rho_H}
The motivating example of an ordinary GMA representation is the $\h[G_{\bQ,S}]$-module $H$ given by the cohomology of modular curves. There is an isomorphism of $\h$-modules $H = H^+ \oplus H^- \risom \h \oplus \h^\vee$, where $\h^\vee$ is the dualizing module of $\h$ (see \cite[\S3.4]{WWE2015}). Because $\End_\h(\h^\vee) \cong \h$, we get a GMA representation $\rho_H : G_{\bQ,S} \ra \Aut_\h(H)$; moreover, it is an ordinary GMA representation with residual pseudorepresentation $\Db$ when the components are ordered as follows:
\[
\End_\h(H) \cong 
\begin{pmatrix}
\End_\h(H^-) & \Hom_\h(H^+, H^-) \\
\Hom_\h(H^-, H^+) & \End_\h(H^+)
\end{pmatrix}
\simeq
\begin{pmatrix}
\h & \h^\vee \\
\Hom_\h(\h^\vee,\h) & \h
\end{pmatrix}.
\]
\end{eg}

Here is a summary of the results of \cite{WWE2015} on ordinary GMA representations.
\begin{prop}
\label{prop:universal ordinary C-H}
\begin{enumerate}
\item There is a \emph{universal ordinary GMA} $E^\ord_\Db$, a quotient of $E_\Db$, such that a GMA representation $G_{\bQ,S} \to E^\times$ with residual pseudorepresentation $\Db$ is ordinary if and only if its map $E_\Db \ra E$ factors through $E^\ord_\Db$. 
\item There is a \emph{universal reducible ordinary GMA} $E^\red_\Db$, a quotient of $E_\Db$, such that a GMA representation $G_{\bQ,S} \to E^\times$  with residual pseudorepresentation $\Db$ is reducible ordinary if and only if its map $E_\Db \ra E$ factors through $E^\red_\Db$.
\item $\End_\h(H)$ has a unique $\h$-GMA structure such that $\rho_H : G_{\bQ,S} \ra \Aut_\h(H)$ is an ordinary GMA representation. 
\end{enumerate}
\end{prop}

\begin{proof}
Statement (1) comes from Proposition 5.9.5, (2) comes from Proposition 7.3.1, and (3) comes from Theorem 7.1.2 of \cite{WWE2015}. 
\end{proof}

\subsection{Ordinary pseudorepresentations}

Having established a notion of ordinary GMA representation, we can now define ordinary pseudorepresentations.

\begin{defn}
\label{defn:ord_PsR}
Let $D: G_{\bQ,S} \ra A$ be a pseudorepresentation deforming $\Db$. Then we call $D$ \emph{ordinary} if there exists an ordinary GMA representation $\rho: G_{\bQ,S} \ra E^\times$ with scalar ring $A$ such that $D = \psi_{\GMA}(\rho)$.
\end{defn}

We write $R^\ord_\Db$ for the scalar ring of $E_\Db^\ord$, which admits a pseudorepresentation $D^\ord_\Db : G_{\bQ,S} \ra R^\ord_\Db$ defined as the composition of $D^u_\Db: G_{\bQ,S} \ra R_\Db$ with $R_\Db \rsurj R^\ord_\Db$. We have shown in \cite[Thm.\ 5.10.4]{WWE2015} that the ring $R_\Db^\ord$ represents the functor of ordinary pseudodeformations of $\Db$, with universal object $D^\ord_\Db$. We write $\cJ \subset R^\ord_\Db$ for the reducibility ideal of $D^\ord_\Db$. 

\begin{rem}
The reason for introducing GMA representations is to make Definition \ref{defn:ord_PsR}: not every pseudodeformation of $\Db$ comes from a representation, but every pseudodeformation comes from a GMA representation. 
\end{rem}

By definition, we see that the modular pseudorepresentation $\psi_{\GMA}(\rho_H): G_{\Q,S} \to \h$ arising from $\rho_H$ (as defined in Example \ref{eg:rho_H}) is ordinary. It can be extended to an $\fH$-valued pseudorepresentation with the following properties. 

\begin{prop}
\label{prop:Rord to fH}
There is a pseudorepresentation $D_H: G_{\bQ,S} \ra \fH$ that is ordinary, deforms $\Db$, and satisfies $D_H \otimes_\fH \h = \psi_{\GMA}(\rho_H)$. The corresponding map $\phi: R^\ord_\Db \ra \fH$ is:
\begin{enumerate}
\item a map of augmented $\Lam$-algebras, where the augmentation ideals are the reducibility ideals $\cJ \subset R^\ord_\Db$ of $D^\ord_\Db$ and $\I \subset \fH$ of $D_H$, and
\item surjective. 
\end{enumerate}
\end{prop}

\begin{proof} 
The pseudorepresentation $D_H$ is constructed by gluing $\psi_{\GMA}(\rho_H)$ together with the Eisenstein pseudorepresentation, and it follows that $D_H \otimes_\fH \h = \psi_{\GMA}(\rho_H)$ and that the reducibility ideal is $\I \subset \fH$ \cite[Cor.\ 7.1.3]{WWE2015}. 

Then (2) follows from \cite[Lem.\ 7.1.4]{WWE2015}, and $(1)$ follows from the fact that $R^\ord_\Db/\cJ \cong \Lam$ \cite[Prop.\ 7.3.1]{WWE2015}.
\end{proof}

Consequently, the functor of reducible ordinary pseudorepresentations is represented by $\Lam$, and $\Lam$ is the scalar ring of $E^\red_\Db$. 

\section{New Results}
\label{sec:new}

With the overview of \cite{WWE2015} complete, we now prove the main theorems.

\subsection{Reducible representations and class groups}
\label{subsec:redrep}

Let us write $E^\ord_\Db$ and $E^\red_\Db$ in GMA form as
\[
E^\ord_\Db \cong \begin{pmatrix}
R^\ord_\Db & B^\ord \\
C^\ord & R^\ord_\Db
\end{pmatrix}
, \qquad
E^\red_\Db \cong \begin{pmatrix}
\Lambda & B^\red \\
C^\red & \Lambda
\end{pmatrix}.
\]

Our goal is to control $\cJ/\cJ^2$ using Galois cohomology. For this, we use the surjection $B^\ord \otimes_{R^\ord_\Db} C^\ord \rsurj \cJ$ discussed in \S\ref{subsec:reducibility}. We also use the fact that the map $E^\red_\Db \ra E^\ord_\Db/\cJ E^\ord_\Db$, which exists because the target receives a reducible ordinary GMA representation, is surjective \cite[Prop.\ 7.3.1(4)]{WWE2015}. (Actually, it is an isomorphism  but that will not be used here.) Composing these surjections, we have 
\begin{equation}
\label{eq:key_surj}
B^\red \otimes_\Lam C^\red \rsurj B^\ord/\cJ \otimes_\Lam C^\ord/\cJ \rsurj \cJ/\cJ^2
\end{equation}
(cf.~\cite[Prop.\ 5.7.2]{WWE2015}). The main result is the following.
\begin{prop}
\label{prop:BandC}
$E^\red_\Db$ is determined as follows.
\begin{enumerate}
\item There exists a natural isomorphism 
\[
X_\chi(1)\lrisom B^\red. 
\]
\item Assume that $X_\theta = 0$. Then there exists a natural isomorphism
\[
\X_\chii^\#(1) \lrisom C^\red.
\]
Moreover, $C^\red$ is free of rank $1$ over $\Lam$.
\end{enumerate}
\end{prop}

First, some lemmas. We will need the following notation, which is particular to \S\ref{subsec:redrep}. We will abbreviate $H^i(\Z[1/Np],-)$ and $H_{(c)}^i(\Z[1/Np],-)$ to $H^i(-)$ and $H_{(c)}^i(-)$, respectively. For $n \in \bZ_{\geq 1}$, we will let $H_{n}^i(-)$ denote $\bigoplus_{\ell | n} H^i(\Q_\ell,-)$, where $\ell$ runs over prime divisors of $n$. Likewise, for a $\Lam$-module $M$ (with trivial $G_{\bQ,S}$-action), we will write $M^\#$ for $\Lam^\# \otimes_\Lam M$ and write $M^\dia{-}$ for $\Lam^\dia{-} \otimes_\Lam M$.
\begin{lem}
\label{lem: no cohom at ell}
For any finitely generated $\Lam$-module $M$ and any $\ell \mid N$ we have $H^1_\ell(M^{\dia{-}}(-1))=0$. In particular, $H^1_{Np}(M^{\dia{-}}(-1))=H^1_p(M^{\dia{-}}(-1))$.
\end{lem}
\begin{proof}
Let $\bar{M}=M/\m_\Lambda M$, where $\m_\Lambda$ is the maximal ideal of $\Lambda$. By Nakayama's lemma, it suffices to show that $H^1_\ell(\bar{M}^{\dia{-}}(-1))=0$. By the Euler characteristic formula and local Tate duality, we have
\begin{equation}
\label{eq: euler char}
\dim (H^1_\ell(\bar{M}^{\dia{-}}(-1))) = \dim(H^0_\ell(\bar{M}^{\dia{-}}(-1))) +\dim(H^0_\ell(\bar{M}^{\#}(2))).
\end{equation}
The Galois action on $\bar{M}^{\dia{-}}(-1)$ and $\bar{M}^{\#}(2)$ is via the characters $\theta \omega^{-1}$ and $\theta^{-1}\omega^2$, respectively. Both these characters are non-trivial at $\ell$ because $\omega$ is trival at $\ell$ and we assume that $\theta$ is primitive. This implies that the $H^0_\ell$ groups appearing in \eqref{eq: euler char} are $0$, so $H^1_\ell(\bar{M}^{\dia{-}}(-1))=0$.
\end{proof}

\begin{lem}
\label{lem:BC_isoms}
Functorially in finitely generated $\Lam$-modules $M$, we have isomorphisms
\begin{equation}
\label{eq:B}
\Hom_\Lam(B^\red, M) \lrisom \Hc^1(M^\dia{-}(-1))
\end{equation}
and
\begin{equation}
\label{eq:C}
\Hom_\Lam(C^\red, M) \lrisom H^1(M^\#(1)).
\end{equation}
\end{lem}

\begin{proof}
For \eqref{eq:C}, \cite[Thm.\ 1.5.5]{BC2009} tells us that there is a natural $\Lam$-linear injective map $i_C: \Hom_\Lam(C^\red, M) \rinj H^1(M^\#(1))$ when $M$ is a cyclic module. But nothing about the proof depends upon $M$ being cyclic, so we have injectivity in general. 

An element of $H^1(M^\#(1))$ results in a short exact sequence of $\Lam[G_{\bQ,S}]$-modules $0 \ra M^\# \ra \mathcal{E} \ra \Lam(-1) \ra 0$. Choose an element $x \in \mathcal{E}$ mapping to $1 \in \Lam$ and write $f: G_{\bQ,S} \ra M$ for the map $\gamma \mapsto \gamma \cdot x - \kcyc^{-1}(\gamma)x$. Then we have a GMA representation
\[
\rho: G_{\bQ,S} \lra \begin{pmatrix} \Lambda & 0 \\ M & \Lambda
\end{pmatrix}, 
\qquad
\rho(\gamma) = \begin{pmatrix} \kcyc^{-1}(\gamma) & 0 \\ f(\gamma) & \langle\gamma\rangle^{-1} \end{pmatrix}.
\]
By the universal property of $E^\red_\Db$, there exists a unique map $C^\red \ra M$ induced by this representation. This construction is an inverse to the construction of $i_C$ in \cite[Thm.\ 1.5.5]{BC2009}, so we have proved that \eqref{eq:C} is an isomorphism. 

Recall the definition of $H^1_{(c)}$ as the cohomology of a mapping cone. Note that $H^0_{Np}(M^\dia{-}(-1)) = 0$ for any $M$. This follows from the fact that the residual character $\chi$ is non-trivial on decomposition groups at all primes dividing $Np$. Consequently, $H^1_{(c)}(M^\dia{-}(-1))$ is naturally isomorphic to the subset of $H^1(M^\dia{-}(-1))$ whose restriction to $H^1_{Np}(M^\dia{-}(-1))$ is zero. By Lemma \ref{lem: no cohom at ell}, $H^1_{Np}(M^\dia{-}(-1))=H^1_{p}(M^\dia{-}(-1))$.

We have an injection
 \[
i_B: \Hom_\Lam(B^\red, M) \rinj H^1_{(c)}(M^\dia{-}(-1)) \subset H^1(M^\dia{-}(-1))
\]
like $i_C$ above, because any extension of $\Lam^{\langle-\rangle}$ by $M(-1)$ realized by a GMA map
\[
E^\red_\Db \ra \begin{pmatrix} \Lam & M \\ 0 & \Lam \end{pmatrix}
\]
induces a trivial extension of $G_p$-representations. The same argument as above shows that $i_B$ is surjective. 
\end{proof}

Now, we prove Proposition \ref{prop:BandC}.

\begin{proof}
We will use the Yoneda lemma for finitely generated $\Lam$-modules to determine $B^\red$ and $C^\red$. Indeed, $B^\red$ and $C^\red$ are finitely generated by the fact that $E_\Db$ is a finite $R_\Db$-module (see \S\ref{subsec:universality}), and the construction of $E^\red_\Db$.

Let's begin with $B^\red$. Let $M$ be an arbitrary finitely generated $\Lambda$-module. We will use the spectral sequence of Proposition \ref{prop:Lambda Poitou-Tate duality} 
\[
\Ext^p_\Lam(H^{3-q}(\Lam^\#(2)), M) \implies \Hc^{p+q}(M^\dia{-}(-1)).
\]
From \cite[Cor.\ 6.3.1]{WWE2015}, we have
\[
H^0(\Lam^\#(2)) \cong 0, \qquad H^1(\Lam^\#(2)) \cong 0, \qquad H^2(\Lam^\#(2)) \cong X_\chi(1). 
\]
The spectral sequence degenerates to yield a functorial isomorphism 
\[
\Hom_\Lambda( X_\chi(1),M)\lrisom \Hc^{1}(M^\dia{-}(-1)).
\] From this and \eqref{eq:B}, we see that $X_\chi(1)$ and $B^\red$ represent the same functor $M \mapsto \Hc^{1}(M^\dia{-}(-1))$. Then the Yoneda lemma  implies that $X_\chi(1) \isoto B^\red$. 

We now similarly calculate $C^\red$. Let $M$ be any finitely generated $\Lambda$-module, and consider the spectral sequence
\[
\Ext^p_\Lam(\Hc^{3-q}(\Lam^\lr), M) \implies H^{p+q}(M^\#(1)).
\]
As in the proof of Proposition \ref{prop:X_control}, we have $\Hc^2(\Lam^\lr) \cong \X_\chii^\#(1)$. Similar arguments show that $\Hc^3(\Lam^\lr)=0$, and, using the weak Leopoldt conjecture (see \cite[Theorem 10.3.22]{NSW2008}), that $\Hc^1(\Lam^\lr) = 0$. Then the spectral sequence degenerates to yield $\Hom_\Lambda(\X_\chii^\#(1), M)\cong H^{1}(M^\#(1))$. As above, using \eqref{eq:C} and the Yoneda lemma, we obtain $\X_\chii^\#(1) \isoto C^\red$. 
\end{proof}

\subsection{A version of Wiles's numerical criterion}
Having controlled $\cJ/\cJ^2$ in terms of Iwasawa class groups, we will now make use of a version of Wiles's numerical criterion \cite[Appendix]{wiles1995} to prove our $R=T$-theorem. We thank Eric Urban for suggesting that the numerical criterion might be used to improve an earlier version. We follow the exposition of \cite{lci}. 

Consider the diagram
\[\xymatrix{
R^\ord_\Db \ar[r]^{\phi} \ar[dr]_{\pi_{R^\ord_\Db}} & \fH \ar[d]^{\pi} \\
 & \Lambda
}\]
where $\phi$ arises from Proposition \ref{prop:Rord to fH}. In this situation, de Smit, Rubin and Schoof prove the following theorem on the way to giving a proof of Wiles's criterion.
\begin{thm} \cite[Thm.\ of \S3 (pg.\ 9)]{lci}
\label{thm: criterion}
The map $\phi$ is an isomorphism of complete intersections if and only if $\phi(\Fitt_{R^\ord_\Db}(\cJ)) \not \subset \m_\Lambda 	\fH$.
\end{thm}

We will use this theorem to show that $\phi$ is an isomorphism under the assumption that $X_\theta=0$. The proof is inspired by the proof of Criteria I in \cite{lci}. We first prove some preliminary results. The notation `$\Fitt$' that appears in the theorem refers to Fitting ideals, which are reviewed in \cite[\S1]{lci}. We will make frequent use of the following well known properties of Fitting ideals (see \cite[Prop.\ 1.1]{lci}).

\begin{lem}
\label{lem: fitting}
Let $A$ be a ring, $M$ a finitely presented $A$-module, and $B$ an $A$-algebra. Then:
\begin{enumerate}
\item $\Fitt_B(M\otimes_A B)=\Fitt_A(M) \cdot B$.
\item $\Fitt_A(M) \subset \Ann_A(M)$.
\end{enumerate}
\end{lem}

Before stating the next lemma, we summarize some known results about the Eisenstein ideals $\I$ and $I$ (see \cite[Prop.\ 3.2.5 and Lem.\ 3.2.9]{WWE2015}): the natural map $\I \to I$ is an isomorphism, $\h/I \cong \Lambda/\xi$, and $\Ann_\fH(\I)=\ker(\fH \to \h)$. In particular, $I$ is a faithful $\h$-module.

\begin{lem}
\label{lem: fitt of I}
We have $\pi(\Fitt_\fH(\I))=\Fitt_\Lambda(I/I^2) \subset (\xi)$ as ideals of $\Lambda$.
\end{lem}
\begin{proof}
By Lemma \ref{lem: fitting}(1), we have 
\[
\pi(\Fitt_\fH(\I)) = \Fitt_\Lambda(\I \otimes_\fH \Lambda).
\]
Since $\I \otimes_\fH \Lambda = \I/\I^2 \cong I/I^2$, we have $\pi(\Fitt_\fH(\I)) = \Fitt_\Lambda(I/I^2)$. 

Since $I$ is a faithful $\h$-module, Lemma \ref{lem: fitting}(2) implies that $\Fitt_\h(I)=0$. Applying Lemma \ref{lem: fitting}(1), we have $\Fitt_{\h/I}(I/I^2)=0$. Recalling that $\h/I \cong \Lambda/\xi$, another application of the same lemma gives $\Fitt_{\Lambda}(I/I^2) \subset (\xi)$. 
\end{proof}

\begin{prop}
\label{prop: X=I/I2}
Assume that $X_\theta=0$. Then the $\Lambda$-modules $X_\chi(1), \cJ/\cJ^2,\I/\I^2$, and $I/I^2$ are all isomorphic and they all have Fitting ideal over $\Lambda$ equal to $(\xi)$.
\end{prop}
\begin{rem}
Sharifi has studied a map from class groups to $I/I^2$ similar to the one that appears in the following proof; see the map of \cite[Thm.\ 5.2]{sharifi2007}. 
\end{rem}
\begin{proof}
From \eqref{eq:key_surj}, we have a surjections
\begin{equation}
\label{eq: BC to I}
B^\red \otimes_\Lambda C^\red \onto \cJ/\cJ^2 \onto \I/\I^2 \isoto I/I^2.
\end{equation}
By Proposition \ref{prop:BandC}, we have $B^\red \cong X_\chi(1)$ and $C^\red \cong \X_\chii^\#(1)$, and so by Proposition \ref{prop:X_control}, we have $C^\red \simeq \Lambda$. Hence we have a surjection
\[
\Theta: X_\chi(1) \onto I/I^2.
\]
But by the previous lemma, $\Fitt_\Lambda(I/I^2) \subset (\xi) = \mathrm{char}_\Lambda(X_\chi(1))$. This implies that the $\ker(\Theta)$ is finite (see \cite[Lem.\ A.7]{wake2} for example), and hence $\ker(\Theta)=0$, since $X_\chi(1)$ has no finite submodule by Ferrero-Washington \cite{FW1979}. Thus $\Theta$ is an isomorphism, and therefore so are all of the maps in \eqref{eq: BC to I}.

Since the modules are all isomorphic, it suffices to compute $\Fitt_\Lambda(X_\chi(1))$, which is well-known to be $(\xi)$. Indeed, by Ferrero-Washington and Lemma \ref{lem: Lambda-modules}, $X_\chi(1)$ has projective dimension 1. Therefore the ideal $\Fitt_\Lambda(X_\chi(1))$ is principal. This implies $\Fitt_\Lambda(X_\chi(1))=\mathrm{char}_\Lambda(X_\chi(1))$ (see \cite[Lem.\ A.6]{wake2}, for example) which is $(\xi)$ by definition.
\end{proof}

\begin{lem}
\label{lem: AnnI}
We have $\Ann_\fH(\I)=\ker(\fH \onto \h)$, and, in particular, $\Ann_\fH(\I) \not \subset \m_\Lambda\fH$. The restriction of $\pi$ to $\Ann_\fH(\I)$ induces an isomorphism
\[
\pi|_{\Ann}:\Ann_\fH(\I) \isoto (\xi).
\]
\end{lem}
\begin{proof}
That $\Ann_\fH(\I)=\ker(\fH \onto \h)$ and the fact that $\pi|_\Ann$ is an isomorphism follows from \cite[Prop.\ 3.2.5]{WWE2015}. To see that $\Ann_\fH(\I) \not \subset \m_\Lambda\fH$, note that $\fH \onto \h$ is a surjection of free $\Lambda$-modules of distinct rank, and so $\Ann_\fH(\I)=\ker(\fH \onto \h)$ is a non-zero $\Lambda$-free direct summand of $\fH$.
\end{proof}

\begin{thm}
\label{thm:R=T}
Assume that $X_\theta=0$. Then $\phi: R^\ord_\Db \to \fH$ is an isomorphism of complete intersections.
\end{thm}
\begin{proof}
By Theorem \ref{thm: criterion} it suffices to show that 
\[
\phi(\Fitt_{R^\ord_\Db}(\cJ)) \not \subset \m_\Lambda\fH.
\]

By Lemma \ref{lem: fitting}(1) and Proposition \ref{prop: X=I/I2}, we have
\[
\pi_{R^\ord_\Db}(\Fitt_{R^\ord_\Db}(\cJ)) =  \Fitt_\Lambda(\cJ \otimes_{R^\ord_\Db} \Lambda)=  \Fitt_\Lambda(\cJ/\cJ^2) = (\xi).
\]

On the other hand, Lemma \ref{lem: fitting}(2) implies $\Fitt_{R^\ord_\Db}(\cJ) \subset \Ann_{R^\ord_\Db}(\cJ)$, and, since $\phi|_\cJ: \cJ \to \I$ is surjective, we have $\phi(\Ann_{R^\ord_\Db}(\cJ)) \subset \Ann_\fH(\I)$. This implies that $\phi(\Fitt_{R^\ord_\Db}(\cJ)) \subset \Ann_\fH(\I)$. Now we have 
\[
\pi_{R^\ord_\Db}(\Fitt_{R^\ord_\Db}(\cJ)) = \pi(\phi(\Fitt_{R^\ord_\Db}(\cJ))) = \pi|_{\Ann}(\phi(\Fitt_{R^\ord_\Db}(\cJ))).
\]
But we know that $\pi_{R^\ord_\Db}(\Fitt_{R^\ord_\Db}(\cJ))=(\xi)$, and $\pi(\Ann_\fH(\I))=(\xi)$, so we have
\[
\pi|_{\Ann}(\phi(\Fitt_{R^\ord_\Db}(\cJ))) = \pi|_{\Ann}(\Ann_\fH(\I)).
\]
Since $\pi|_{\Ann}$ is an isomorphism by Lemma \ref{lem: AnnI}, we conclude that $\phi(\Fitt_{R^\ord_\Db}(\cJ))=\Ann_\fH(\I)$. It then follows from Lemma \ref{lem: AnnI} that $ \phi(\Fitt_{R^\ord_\Db}(\cJ)) \not \subset \m_\Lambda\fH$.
\end{proof}
\begin{rem}
This answers in the affirmative a question of Sharifi \cite[\S5]{sharifi2009} whether $\fH$ is Gorenstein when $X_\theta=0$. As noted in \emph{loc.~cit.}, it follows from work of Ohta that $\fH$ is Gorenstein when $\X_\theta=0$, and so our result improves on Ohta's. It was proven in \cite[Thm.\ 1.2]{wake1} that $\fH$ is Gorenstein when $X_\theta=0$ under the additional assumption of Sharifi's conjecture.
\end{rem}

\section{Applications} 
\label{sec:app}

\subsection{Toward Sharifi's conjecture} We have the following immediate corollary. 

\begin{cor}
\label{cor: principal}
Assume that $X_\theta = 0$ and that $X_\chi(1)$ is cyclic. Then the ideals $\cJ \subset R^\ord_\Db$, $\I \subset \fH$ and $I \subset \h$ are all principal, and both $\fH$ and $\h$ are complete intersections.
\end{cor}
\begin{proof}
By Proposition \ref{prop: X=I/I2} we have $X_\chi(1) \simeq \cJ/\cJ^2 \cong \I/\I^2 \cong I/I^2$, and so, if $X_\chi(1)$ is cyclic, then each ideal is principal by Nakayama's lemma. Since $I$ is faithful as an $\h$-module, it must be generated by a non-zero divisor. Then since $\h/I \cong \Lambda/\xi$ is complete intersection and $I$ is generated by a regular sequence, $\h$ is complete intersection (see \cite[Thm.\ 2.3.4(a), pg.\ 75]{BH1993}, for example). Theorem \ref{thm:R=T} implies that $\fH$ is complete intersection.
\end{proof}

These results have applications to Sharifi's conjecture. Recall that Sharifi's conjecture states that two maps $\Upsilon$ and $\varpi$ are isomorphisms \cite{sharifi2011}. In \cite{FK2012, fks2014} this conjecture was refined to state that $\Upsilon$ and $\varpi$ are mutually inverse. (See also \cite[\S8.1]{WWE2015} for a review of Sharifi's conjecture using the same notation as this paper.)

\begin{cor}
\label{cor: sharifi}
Consider the maps
\[
\Upsilon: X_\chi(1) \to H^-/IH^- \text{ and } \varpi: H^-/IH^- \to X_\chi(1) 
\]
defined by Sharifi. If $X_\theta = 0$, then $\Upsilon$ is an isomorphism. If, in addition, $X_\chi(1)$ is cyclic, then $\varpi$ is an isomorphism as well. Finally, if, in addition, $\xi_\chi$ has no multiple root, then they are mutual inverses.
\end{cor}
\begin{proof}
By Theorem \ref{thm:R=T}, if $X_\theta = 0$, then $\fH$ is Gorenstein. It is known that if $\fH$ is Gorenstein, then $\Upsilon$ is an isomorphism (see \cite[Prop.\ 4.10]{sharifi2011}). If $X_\chi(1)$ is cyclic, then the previous corollary implies that $\fH$ and $\h$ are complete intersections, and hence Gorenstein. The result now follows by work of Fukaya-Kato.

Indeed, since $\h$ is Gorenstein and $H^-$ is a dualizing module over $\h$ (see \cite[Cor.\ 3.4.2]{WWE2015}), we see that $H^-/IH^- \simeq \h/I \cong \Lambda/\xi$, which has no $p$-torsion. Moreover, the fact that $\h$ is Gorenstein implies the condition $C(\h)$ of Fukaya-Kato by \cite[\S7.2.10]{FK2012}. Then the final two claims follow from Theorems 7.2.8 and 7.2.7 of \emph{loc.\ cit.}, respectively. 
\end{proof}

\subsection{Residually reducible modularity}
\label{subsec:Skinner-Wiles}
In this subsection, we deduce from the main Theorem \ref{thm:R=T} a new proof of a known modularity result for ordinary Galois representations. This may be viewed as a verification that our notion of ordinary GMA representation in \S\ref{subsec:ord_GMA_reps} can be usefully applied to ordinary Galois representations as they are usually known. We emphasize that from now on we use the word ``ordinary'' exclusively to refer to Definition \ref{defn:ord_C-H}. 

\begin{rem}
\label{rem:conv_suspend}
Unlike the rest of this paper, we are not assuming the running assumptions about $\theta$, $p$, and $N$ from \S\ref{subsec:iwasawa} in this subsection. Rather, we will verify that the running assumptions follow from the assumptions of Theorem \ref{thm: our sw}. 
\end{rem}

One of the main results of Skinner and Wiles' work \cite{SW1999} is 
\begin{thm}[Skinner-Wiles]
\label{thm:sw}
Suppose that $\rho: G_{\bQ} \ra \Aut_F(V)$ is continuous, irreducible, and ramified at finitely many primes, where $V$ is a 2-dimensional $F$-vector space, $F/\bQ_p$ is a finite extension, and $p$ is odd. Suppose that $\bar\rho^{ss} \simeq \omega^{-1} \oplus \theta^{-1}$ and that 
\begin{enumerate}
\item $\theta \vert_{G_p} \neq \omega \vert_{G_p}$,
\item $\rho \vert_{I_p}$ is conjugate to $\begin{pmatrix} \kcyc^{-1}\vert_{I_p} & 0 \\ * & *\end{pmatrix}$, 
\item $\det\rho = \tau \kcyc^{k-3}$ is odd, where $\tau$ is a finite order character and $k \geq 2$.
\end{enumerate}
Then $\rho$ comes from a modular form.
\end{thm}
\begin{rem}
Our conventions for modular forms and their Galois representations follow those of \cite{FK2012}, and differ slightly from those of \cite{SW1999}. For us, the phrase \emph{$\rho$ comes from a modular form} means that $\rho$ is isomorphic to $H\otimes_\h F$ where the map $\h \to F$ is induced by an cuspidal eigenform $f$ with coefficients in $F$. By \cite[1.2.9, 1.5.8]{FK2012}, we see that if $\rho$ is modular, then the weight of $f$ must be the integer $k$ of (3).

The representation $\rho$ comes from a modular form in our sense if and only if $\rho \otimes \kcyc$ comes from a modular form in the sense of \cite{SW1999}. With this in mind, the theorem is a restatement of the theorem stated on \cite[pg.~6]{SW1999}.
\end{rem}

We can give a new proof of this result in certain cases, following directly from Theorem \ref{thm:R=T}. Among the additional restrictions, the serious ones are 
\begin{enumerate}[label=(\roman*), leftmargin=2.5em] \addtocounter{enumi}{2}
\item that $\rho$ is a lift of ``minimal level'' 
\item that a Vandiver conjecture type condition holds for the relevant isotypic parts of the class group.
\end{enumerate}

Now fix $\rho$ and $\theta$ as in Theorem \ref{thm:sw}. Let $N$ be the prime-to-$p$ part of the conductor of $\theta$, and write $\chi=\omega^{-1}\theta$. We will treat these both as characters of $G_\bQ$ and as Dirichlet characters of modulus $Np$.

\begin{thm}
\label{thm: our sw}
In addition to the conditions of Theorem \ref{thm:sw}, assume that 
\begin{enumerate}[label=(\roman*)]
\item $p \nmid \phi(N)$ and $p \geq 5$;  
\item $\theta$ is ramified at $p$ when $N > 1$; 
\item $\rho$ is ramified only at primes dividing $Np$; 
\item $X_\theta = 0$ and, if $\chi$ is unramified at $p$, assume that $X_{\theta^{-1}\omega^2} = 0$. 
\end{enumerate}
Then $\rho$ comes from a modular form.
\end{thm}

The proof of Theorem \ref{thm: our sw} relies on applying Theorem \ref{thm:R=T}. In order to apply that theorem, we will need to relate condition (2) of Theorem \ref{thm:sw} -- which is the ``ordinary'' condition in \cite{SW1999} -- to our ordinary condition on GMA representations and on pseudorepresentations, established in  in \cite{WWE2015} and reviewed in \S\ref{sec:OP}. Note that condition (1) implies condition (b) of \S \ref{subsec:iwasawa}, so our definition of ordinary applies in this situation. We will see that either $\rho$ is ordinary in our sense, or $\chi$ is unramified at $p$ and $\rho \otimes \chi$ is ordinary in our sense. 

The following lemma is standard, following from Clifford theory. We provide a proof for lack of a reference. 
\begin{lem}
\label{lem:SW ord implies p-reducible}
Given any $\nu$ satisfying the conditions of Theorem \ref{thm:sw}, $\nu\vert_{G_p}$ is reducible. Moreover, there is a unique quotient character $\alpha$ of $\nu|_{G_p}$ such that $\alpha|_{I_p}=\kcyc^{-1}|_{I_p}$.
\end{lem}
\begin{proof}
Let $\alpha_1=\kcyc^{-1}|_{I_p}$ and $\alpha_2= \tau\kcyc^{k-2}\vert_{I_p}$. By conditions (2)-(3) of Theorem \ref{thm:sw}, we have $\nu|_{I_p}^{ss} \simeq \alpha_1 \oplus \alpha_2$, and $\alpha_1$ is a quotient of $\nu|_{I_p}$. Since $k \ge 2$ and $\tau$ has finite order, we see that $\alpha_1$ and $\alpha_2$ are distinct. Moreover, since they are restrictions of distinct characters of $G_p$, they are not Frobenius conjugate.

The result now follows from Clifford theory, which states that the restriction of an irreducible representation to a normal subgroup is semi-simple and consists of conjugate representations. Indeed, when $\nu\vert_{I_p}$ is indecomposable, $\nu\vert_{G_p}$ must be reducible and indecomposable, and we let $\alpha$ be the quotient character of $\nu\vert_{G_p}$. We see that $\alpha|_{I_p}= \alpha_1$. Otherwise, $\nu\vert_{I_p}$ is semi-simple, and, since $\alpha_1$ and $\alpha_2$ are not conjugate, $\nu\vert_{G_p}$ must be reducible and semi-simple. We let $\alpha$ be the summand of $\nu\vert_{G_p}$ satisfying $\alpha|_{I_p}= \alpha_1$.
\end{proof}

We note that our definition of ordinary in Definition \ref{defn:ord_C-H} depends on the residual pseudorepresentation. Below, we will sometimes write ``$\rho$ is ordinary" to mean ``$\rho$ is ordinary as a deformation of $\psi(\bar{\rho})$". Let $\Db=\psi(\omega^{-1} \oplus \theta^{-1})$ and let $\Db'=\psi(\omega^{-1}\oplus \theta'^{-1})$, where $\theta'=\omega\chi^{-1} = \omega^2 \theta^{-1}$. We are now ready to prove the following.

\begin{lem}
\label{lem:prime}
Let $\rho$ be as in Theorem \ref{thm:sw} and also satisfying condition (iii) of Theorem \ref{thm: our sw}. If $\chi$ is ramified at $p$, then $\rho$ is an ordinary deformation of $\Db$. If $\chi$ is unramified at $p$, then either $\rho$ is an ordinary deformation of $\Db$, or $\rho \otimes \chi$ is an ordinary deformation of $\Db'$. 
\end{lem}

\begin{proof}
We will use the criterion for a representation to be ordinary given in Example \ref{eg:ordinary representation}. By Lemma \ref{lem:SW ord implies p-reducible}, we see that $\rho|_{G_p}$ is reducible. We will label the Jordan-H\"older factors in two different ways. First we write $\rho|_{G_p}^{ss}=\alpha_1 \oplus \alpha_2$ with $\alpha_1|_{I_p}=\kcyc^{-1}|_{I_p}$. By the same lemma, we see that $\alpha_1$ is a quotient of $\rho|_{G_p}$.

We also write $\rho|_{G_p}^{ss}=\alpha_a \oplus \alpha_d$ with $\bar{\alpha}_a=\omega^{-1}$ and $\bar{\alpha}_d=\theta^{-1}$. (The labeling is meant to reflect the coordinate labels $\sm{a}{b}{c}{d}$ of $2\times 2$-matrix and the ordering convention of \eqref{eq: GMA defm}.) We see that if $\alpha_1=\alpha_a$ then $\alpha_1$ is a quotient character such that $\bar{\alpha}_1=\omega^{-1}$ and $\alpha_1|_{I_p}=\kcyc^{-1}|_{I_p}$, so $\rho$ is ordinary by Example \ref{eg:ordinary representation}. Moreover, if $\chi$ is ramified at $p$, we see that $\bar{\alpha}_a|_{I_p} \ne \bar{\alpha}_d|_{I_p}$, and so, since $\bar{\alpha}_1|_{I_p} = \bar{\alpha}_a|_{I_p}$, we must have $\alpha_1=\alpha_a$.

Now assume that $\alpha_1=\alpha_d$ (and so $\chi$ is unramified at $p$), and let $\rho'=\rho\otimes \chi$. Since
$\chi$ is unramified at $p$, one can verify that $\rho \otimes \chi$ satisfies the conditions of Theorem \ref{thm:sw}. We apply Lemma \ref{lem:SW ord implies p-reducible} to $\rho'$, and we get $\rho'|_{G_p}^{ss}=\alpha_1' \oplus \alpha_2'$ with $\alpha_1'$ a quotient and $\alpha_1'|_{I_p}=\kcyc^{-1}|_{I_p}$. Since $\chi$ is unramified at $p$, we must have $\alpha_1'=\alpha_1\chi$. We also write $\rho'|_{G_p}^{ss}=\alpha_a' \oplus \alpha_d'$ with $\bar{\alpha}'_a=\omega^{-1} = \theta^{-1}\chi$ and $\bar{\alpha}'_d=\theta'^{-1} =\omega^{-1}\chi$. We see that $\alpha'_a=\alpha_d\chi$. Hence we have 
\[
\alpha_1'= \alpha_1\chi = \alpha_d \chi = \alpha_a',
\]
so $\rho'$ is ordinary by Example \ref{eg:ordinary representation}.
\end{proof}

The following lemma prepares us to apply Theorem \ref{thm:R=T} to prove Theorem \ref{thm: our sw}. \begin{lem}
\label{lem:AS}
Assume that $\theta$, $N$, and $p$ satisfy the conditions imposed by Theorems \ref{thm:sw} and Theorem \ref{thm: our sw}. Then $\theta$, $N$, and $p$ satisfy the running assumptions of \S\ref{subsec:iwasawa}. Also, if $\chi$ is unramified at $p$, then $\theta'=\omega^2 \theta^{-1}$ satisfies these running assumptions. 
\end{lem}

\begin{proof}
The conditions $p \nmid N \phi(N)$ and $p \geq 5$ imposed in \S\ref{subsec:iwasawa} are implied by condition (i) of Theorem \ref{thm: our sw} and the definition of $N$ as the tame level of $\theta$. 

As we have already noted, $\theta$ and $\theta'$ satisfy condition (1) of Theorem \ref{thm:sw}, which is equivalent to condition (b) of \S\ref{subsec:iwasawa}. Thus it remains to verify conditions (a) and (c). 

Suppose that $N=1$. We claim that the existence of $\rho$ implies that $\theta$ cannot be $1$ or $\omega^2$. Indeed, in these cases, $\chi = \omega^{\pm 1}$, and Stickelberger's theorem implies that $X_\chi = 0$ (see e.g.~\cite[Prop.\ 6.16 and Thm.\ 6.17, pg.\ 102]{washington1982}, and note that $B_2=\frac{1}{6}$). However, \cite[Prop.\ 2.1]{ribet1976} implies that the irreducible representation $\rho$ leaves stable a lattice such that the resulting residual representation $\bar{\rho}$ is not diagonalizable and has $\omega^{-1}$ as a subrepresentation.  Moreover, because $\rho$ satisfies the conditions of Theorem \ref{thm:sw}, and since $\chi$ is ramified at $p$, the argument of Lemma \ref{lem:prime} implies that $\bar{\rho}|_{G_p}$ has $\omega^{-1}$ as a quotient representation. This implies that $\bar{\rho}$ is a non-trivial extension of $\theta^{-1}$ by $\omega^{-1}$ that is split upon restriction to $G_p$. This extension gives rise to a non-zero element of $X_{\chi} \otimes_{\Lam} \F_p$, so $X_{\chi} \ne 0$, a contradiction. Hence, if $N=1$, then $\theta \ne 1$ and $\theta$ is primitive (so (a) is true), and $\theta \neq \omega^2$ (so (c) is true). 

Now suppose $N > 1$, making (c) satisfied. Assumption (i) implies that our running assumption that $p \nmid N\phi(N)$ is satisfied. We have that $\theta$ is primitive of modulus either $N$ or $Np$ by definition of $N$. Assumption (ii) rules out the case that $\theta$ is primitive of modulus $N$, so (a) holds. 

When $\chi$ is unramified at $p$, we wish to show that $\theta'$ satisfies the running assumptions. Since $\chi$ is odd, the assumption that $\chi$ is unramified at $p$ implies that $N>1$. As $\theta'=\omega^{2}\theta^{-1}$, the conductor of $\theta'$ is either $N$ or $Np$. Since $\chi'=\chi^{-1}$ is unramified at $p$, this implies that $\theta'=\chi'\omega$ is ramified at $p$, making its conductor $Np$ and satisfying assumption (a). 
\end{proof}

\begin{proof}[Proof of Theorem \ref{thm: our sw}] In Lemma \ref{lem:AS} we checked that the running assumptions of the paper about $\theta$, $\theta'$, $N$, and $p$ are satisfied. By Lemma \ref{lem:prime}, we can break the proof into two cases. 

\textbf{Case 1: $\rho$ is ordinary.} We consider the case where $\rho$ is ordinary. In this case $\psi(\rho)$ is an ordinary pseudorepresentation, because $\rho$ is an ordinary (GMA) representation inducing it. Therefore there is a unique map $\nu: R^\ord_\Db \ra F$ corresponding to $\psi(\rho)$. Assumption (iv) allows us to apply Theorem \ref{thm:R=T} so that $\phi: R^\ord_\Db \risom \fH$. Then the ordinary $p$-adic modular eigenform $f$ determined by $\nu \circ \phi^{-1} : \fH \ra F$ satisfies $\psi(\rho_f) = \psi(\rho)$, which implies $\rho_f \simeq \rho$ since $\rho$ is irreducible. Condition (3) implies that this modular form has weight $k \in \bZ_{\geq 2}$; consequently, $f$ is classical by \cite[Thm.\ I]{hida1986}, and $\rho$ is modular. 

\textbf{Case 2: $\chi$ is unramified at $p$ and $\rho'$ is ordinary.} We want to apply Theorem \ref{thm:R=T} to $\Db'$ now. Assumption (iv) allows us to apply Theorem \ref{thm:R=T}, so that there is an isomorphism $\phi' : R_{\Db'} \risom \fH'$, where $\fH'$ is $\fH$ for ${\theta'}$. Then we have a unique map $\nu': R^\ord_{\Db'} \ra F$ corresponding to $\psi(\rho')$, and the rest of the argument is the same as above. 
\end{proof}

\begin{rem}
We see in the proof that for each $\rho$, only one of the two groups $X_\theta$ and $X_{\theta^{-1}\omega^2}$ must be zero. Moreover, if $\chi$ is ramified at $p$, we can be certain that it is $X_\theta$ that must be $0$. 

Before \cite{SW1999}, Skinner and Wiles gave a different proof of the modularity of $\rho$ under different hypotheses \cite{SW1997}. Among their assumptions is that $X_\chii = 0$. This is equivalent to $\X_\theta=0$ by the reflection principle, and so it is a much stronger assumption than our assumption that $X_\theta=0$. In this way, Theorem \ref{thm: our sw} may be seen as an improvement of the method of \cite{SW1997}.

\end{rem}

\bibliographystyle{alpha}
\bibliography{CWEbib-0915}

\def\cprime{$'$} \def\Dbar{\leavevmode\lower.6ex\hbox to 0pt{\hskip-.23ex
  \accent"16\hss}D} \def\cfac#1{\ifmmode\setbox7\hbox{$\accent"5E#1$}\else
  \setbox7\hbox{\accent"5E#1}\penalty 10000\relax\fi\raise 1\ht7
  \hbox{\lower1.15ex\hbox to 1\wd7{\hss\accent"13\hss}}\penalty 10000
  \hskip-1\wd7\penalty 10000\box7}
  \def\cftil#1{\ifmmode\setbox7\hbox{$\accent"5E#1$}\else
  \setbox7\hbox{\accent"5E#1}\penalty 10000\relax\fi\raise 1\ht7
  \hbox{\lower1.15ex\hbox to 1\wd7{\hss\accent"7E\hss}}\penalty 10000
  \hskip-1\wd7\penalty 10000\box7}
\begin{thebibliography}{WWE15}

\bibitem[BC09]{BC2009}
Jo{\"e}l Bella{\"{\i}}che and Ga{\"e}tan Chenevier.
\newblock Families of {G}alois representations and {S}elmer groups.
\newblock {\em Ast\'erisque}, (324):xii+314, 2009.

\bibitem[BH93]{BH1993}
Winfried Bruns and J{\"u}rgen Herzog.
\newblock {\em Cohen-{M}acaulay rings}, volume~39 of {\em Cambridge Studies in
  Advanced Mathematics}.
\newblock Cambridge University Press, Cambridge, 1993.

\bibitem[dSRS97]{lci}
Bart de~Smit, Karl Rubin, and Ren{\'e} Schoof.
\newblock Criteria for complete intersections.
\newblock In {\em Modular forms and {F}ermat's last theorem ({B}oston, {MA},
  1995)}, pages 343--356. Springer, New York, 1997.

\bibitem[FK12]{FK2012}
Takako Fukaya and Kazuya Kato.
\newblock On conjectures of {S}harifi.
\newblock Preprint, 2012.

\bibitem[FKS14]{fks2014}
Takako Fukaya, Kazuya Kato, and Romyar Sharifi.
\newblock Modular symbols in {I}wasawa theory.
\newblock In {\em Iwasawa Theory 2012}, volume~7 of {\em Contributions in
  Mathematical and Computational Sciences}, pages 177--219. Springer-Verlag,
  Berlin Heidelberg, 2014.

\bibitem[FW79]{FW1979}
Bruce Ferrero and Lawrence~C. Washington.
\newblock The {I}wasawa invariant {$\mu _{p}$} vanishes for abelian number
  fields.
\newblock {\em Ann. of Math. (2)}, 109(2):377--395, 1979.

\bibitem[Hid86]{hida1986}
Haruzo Hida.
\newblock Galois representations into {${\rm GL}_2({\bf Z}_p[[X]])$} attached
  to ordinary cusp forms.
\newblock {\em Invent. Math.}, 85(3):545--613, 1986.

\bibitem[Jan89]{jannsen1989}
Uwe Jannsen.
\newblock Iwasawa modules up to isomorphism.
\newblock In {\em Algebraic number theory}, volume~17 of {\em Adv. Stud. Pure
  Math.}, pages 171--207. Academic Press, Boston, MA, 1989.

\bibitem[LS13]{LS2013}
Meng~Fai Lim and Romyar~T. Sharifi.
\newblock Nekov\'a\v r duality over {$p$}-adic {L}ie extensions of global
  fields.
\newblock {\em Doc. Math.}, 18:621--678, 2013.

\bibitem[Nek06]{nekovar2006}
Jan Nekov{\'a}{\v{r}}.
\newblock Selmer complexes.
\newblock {\em Ast\'erisque}, (310):viii+559, 2006.

\bibitem[NSW08]{NSW2008}
J{\"u}rgen Neukirch, Alexander Schmidt, and Kay Wingberg.
\newblock {\em Cohomology of number fields}, volume 323 of {\em Grundlehren der
  Mathematischen Wissenschaften [Fundamental Principles of Mathematical
  Sciences]}.
\newblock Springer-Verlag, Berlin, second edition, 2008.

\bibitem[Oht05]{ohta2005}
Masami Ohta.
\newblock Companion forms and the structure of {$p$}-adic {H}ecke algebras.
\newblock {\em J. Reine Angew. Math.}, 585:141--172, 2005.

\bibitem[Rib76]{ribet1976}
Kenneth~A. Ribet.
\newblock A modular construction of unramified {$p$}-extensions of{$Q(\mu
  _{p})$}.
\newblock {\em Invent. Math.}, 34(3):151--162, 1976.

\bibitem[Sha07]{sharifi2007}
Romyar~T. Sharifi.
\newblock Iwasawa theory and the {E}isenstein ideal.
\newblock {\em Duke Math. J.}, 137(1):63--101, 2007.

\bibitem[Sha09]{sharifi2009}
Romyar~T. Sharifi.
\newblock Cup products and {$L$}-values of cusp forms.
\newblock {\em Pure Appl. Math. Q.}, 5(1):339--348, 2009.

\bibitem[Sha11]{sharifi2011}
Romyar Sharifi.
\newblock A reciprocity map and the two-variable {$p$}-adic {$L$}-function.
\newblock {\em Ann. of Math. (2)}, 173(1):251--300, 2011.

\bibitem[SW97]{SW1997}
C.~M. Skinner and A.~J. Wiles.
\newblock Ordinary representations and modular forms.
\newblock {\em Proc. Nat. Acad. Sci. U.S.A.}, 94(20):10520--10527, 1997.

\bibitem[SW99]{SW1999}
C.~M. Skinner and A.~J. Wiles.
\newblock Residually reducible representations and modular forms.
\newblock {\em Inst. Hautes \'Etudes Sci. Publ. Math.}, (89):5--126 (2000),
  1999.

\bibitem[Wak15a]{wake2}
Preston Wake.
\newblock Eisenstein {H}ecke algebras and conjectures in {I}wasawa theory.
\newblock {\em Algebra Number Theory}, 9(1):53--75, 2015.

\bibitem[Wak15b]{wake1}
Preston Wake.
\newblock Hecke algebras associated to {$\Lambda$}-adic modular forms.
\newblock {\em J. Reine Angew. Math.}, 700:113--128, 2015.

\bibitem[Was82]{washington1982}
Lawrence~C. Washington.
\newblock {\em Introduction to cyclotomic fields}, volume~83 of {\em Graduate
  Texts in Mathematics}.
\newblock Springer-Verlag, New York, 1982.

\bibitem[WE15]{WE2015}
Carl Wang~Erickson.
\newblock Algebraic families of {G}alois representations and potentially
  semi-stable pseudodeformation rings.
\newblock To appear in \textit{Math. Ann}. arXiv:1501.05629v2 [math.NT], 2015.

\bibitem[Wei94]{weibel1994}
Charles~A. Weibel.
\newblock {\em An introduction to homological algebra}, volume~38 of {\em
  Cambridge Studies in Advanced Mathematics}.
\newblock Cambridge University Press, Cambridge, 1994.

\bibitem[Wil95]{wiles1995}
Andrew Wiles.
\newblock Modular elliptic curves and {F}ermat's last theorem.
\newblock {\em Ann. of Math. (2)}, 141(3):443--551, 1995.

\bibitem[WWE15]{WWE2015}
Preston Wake and Carl Wang~Erickson.
\newblock Pseudo-modularity and {I}wasawa theory.
\newblock arXiv:1505.05128v2 [math.NT], 2015.

\end{thebibliography}

\end{document}